\documentclass[article,oneside,oldfontcommands,b5paper,10pt]{memoir}

\counterwithout{section}{chapter}
\setsecnumdepth{subsubsection}

\setsecheadstyle{\bf\raggedright}
\setsubsecheadstyle{\sl\raggedright}
\setsubsubsecheadstyle{\sl\raggedright}
\setsecnumformat{\csname the#1\endcsname.\;}

\usepackage[affil-it]{authblk}

\usepackage{indentfirst}
\usepackage[margin=1.05 in]{geometry}

\usepackage[numbers,sort&compress]{natbib}

\usepackage{amssymb,amsmath,mathrsfs}%,mathabx}
\usepackage{amsthm,stmaryrd}
\usepackage[hypertexnames=false,colorlinks,linkcolor={blue},citecolor={blue}]{hyperref}
\usepackage{graphicx,float,caption}

\usepackage{epsfig,subcaption}

\usepackage{scalerel} % used for decreasing the size of \bigcap and \bigcup

\usepackage{enumitem}
\setlist{
leftmargin=*, 
topsep = 4pt, 
itemsep = 1pt, 
}

\usepackage[T1]{fontenc}

\usepackage{fancyhdr}
\fancypagestyle{basic}{
\fancyhf{}
\fancyhead[L]{\thetitle}
\fancyhead[R]{\thepage}
}
\fancypagestyle{draftfirstpage}{

\fancyhf{}
}

\numberwithin{equation}{section}

\newtheorem{theorem}{Theorem}[section]

\newtheorem{lemma}[theorem]{Lemma}

\theoremstyle{definition}
\newtheorem{definition}[theorem]{Definition}

\theoremstyle{remark}

\allowdisplaybreaks

\newcommand{\reDeclareMathOperator}[2]{\let#1\undefined \DeclareMathOperator{#1}{#2}}
\newcommand{\reDeclareMathOperatorL}[2]{\let#1\undefined \DeclareMathOperator*{#1}{#2}}

\reDeclareMathOperator{\mod}{mod}

\reDeclareMathOperator{\supp}{supp}

\reDeclareMathOperator{\Proj}{Proj}
\DeclareMathOperator{\gr}{gr}

\DeclareMathOperator{\fix}{fix}
\DeclareMathOperator{\Endpoints}{end}
\DeclareMathOperator{\EP}{ep}

\renewcommand{\emptyset}{\varnothing}
\renewcommand{\bar}{\overline}

\newcommand{\R}{\mathbb{R}}
\newcommand{\N}{\mathbb{N}}

\renewcommand{\multimap}{\rightrightarrows}

\renewcommand{\tilde}{\widetilde}

\let\originalbigcap\bigcap
\let\originalbigcup\bigcup
\reDeclareMathOperatorL{\bigcap}{\mathbin{\scaleobj{0.9}{\originalbigcap}}}
\reDeclareMathOperatorL{\bigcup}{\mathbin{\scaleobj{0.9}{\originalbigcup}}}

\makeatletter
\renewcommand*{\@fnsymbol}[1]{\ifcase#1\or*\else\@arabic{#1}\fi}
\makeatother

\title{Variational principles using a non-symmetric non-triangular distance}

\author[1]{Natthaya Boonyam}
\author[1,2]{Parin Chaipunya\thanks{Corresponding author}$^{,}$}
\author[1,2]{Poom Kumam}

\affil[1]{
{Department of Mathematics, Faculty of Science, King Mongkut's University of Technology Thonburi},
{126 Pracha Uthit Rd.},
{Bang Mod, Thung Khru, Bangkok},
{10140}, 
{Thailand.}
}
\affil[2]{
{Center of Excellence in Theoretical and Computational Science (TaCS-CoE), King Mongkut's University of Technology Thonburi},
{126 Pracha Uthit Rd.},
{Bang Mod, Thung Khru, Bangkok},
{10140}, 
{Thailand.}

Email: natthaya.b@mail.kmutt.ac.th, parin.cha@kmutt.ac.th and  poom.kum@kmutt.ac.th;
\vspace{.25cm}
}

\date{}

\begin{document}
\maketitle \vspace{-1.8cm}
\thispagestyle{draftfirstpage}

\newcommand{\pre}{\preccurlyeq_{\varepsilon,\varphi}}

\begin{abstract}
  We consider Borwein-Preiss and Ekeland variational principles using distance functions that neither is symmetric nor enjoy the triangular inequality.
  All the given results rely exclusively on the convergence and continuity behaviors induced synthetically by the distance function itself without any topological implications.
  At the end of the paper, we also present two applications; the Caristi fixed point theorem and an existence theorem for equilibrium problems.

	%The motive of this article is to expose the significance of Non-triangular metric in the context of generalization of metric fixed point theorems.
	
	%	\medskip
	
	
	%\noindent{\bf 2010 MSC:} 47H09, 47H10, 47H20, 47H07.
\end{abstract}

%\footnotesize
%\tableofcontents*
%\normalsize

\normalsize

\section{Introduction}

Variational principles have been one of the cornerstones of the contemporary variational analysis.
The concept was initiated by \citet{zbMATH03449362} where he proved his celebrated theorem known today as the \emph{Ekeland variational principle}.
The classical Weierstra\ss{} theorem guarantees a global minimum of a lower semicontinuous function on a compact set, while the Ekeland variational principle says that a lower semicontinuous function that is bounded below could be \emph{perturbed} ever so slightly, with a metric function, to attain a minimum.
It also provides a geometric insight of a global minimum.
When a function has a global minimum, one could construct a hyperplane that supports the whole graph at the minimizer.
On the other hand, the Ekeland variational principle asserts that we could always use a cone of a given \emph{width} to support the graph of any lower semicontinuous function that is bounded below near its approximate minimum.

% Other forms of variational principles also exist in the literature.
One downside of the Ekeland variational principle is that the perturbation kernel, which is the metric function, is not differentiable.
The smooth kernel alternative was proposed in \cite{zbMATH04028246}, where the so-called \emph{Borwein-Preiss variational principle} was proved for Banach spaces.
The Borwein-Preiss variational principle is not a generalization of the Ekeland variational principle and it is not possible to deduce one from another.
The Borwein-Priess variational principle was generalized into metric spaces using gauge-type functions by \citet{zbMATH01496358}.
Eventough there are other forms of variational principles ({\itshape e.g.}~\cite{zbMATH03227405}), the two mentioned above  were the most attended ones.
To list few examples, the Ekeland variational principle were generalized into quasi-metric spaces~\cite{zbMATH07103811, zbMATH05888209}, $b$-metric spaces~\cite{zbMATH05871143}, partial metric spaces~\cite{zbMATH07271705}.
In~\cite{ESHGHINEZHAD20131250} and~\cite{zbMATH07752868}, generalized distances were used to established a version of Ekeland variational principle but it is interesting to note that both still require the triangle inequality to be satisfied.
In the setting of a Banach space, \citet{zbMATH01736511} proposed a variational principle of Borwein-Preiss type in which the perturbation kernel is not necessarily a norm.
The Borwein-Preiss variational principle was vastly refined by \citet{zbMATH06522706} in the framework of a metric space.
It should be noted that even when a non-trinangular function is used in a metric space setting, most of the time~\cite{zbMATH07752868, zbMATH06187183, zbMATH01496358, KHANH20102245} the triangular inequality of a metric is still involved {\itshape e.g.} in the form of Cantor's intersection theorem.

In this paper, we generalize variational principles {\itshape \`{a} la} Borwein-Preiss and Ekeland using a distance function that is not necessarily be symmetric nor satisfy triangle inequality.
The assumptions of the two variational principles are based only on the sequential behaviors induced by the given distance function.
One would see that the proofs are based on a relaxed form of Cantor's intersection theorem.
We also present two applications of the our variational principles. 
In the first application, the Caristi fixed point theorem using non-symmetric and non-triangular distances was proved.
Finally, we adopt the compactness based on the distance function~$d$ and prove an existence result for an equilibrium problem.

\section{Definitions and lemmas}

We consider in this section notation and preliminary results that are used in the subsequent sections.
We begin with a very general approach to a distance function on any nonempty set.
\begin{definition}
  Given a nonempty set~$X$.
  A function~$d : X \times X \to [0,\infty)$ is called a \emph{distance function} if for any~$x,y \in X$, $d(x,y) = 0 \iff x = y$.
\end{definition}
The definition of a distance function is so general that an explicit example is not needed.
In particular, it includes a metric space, quasi-metric space~\cite{zbMATH07900262, zbMATH07103811, zbMATH05888209}, $b$-metric space~\cite{zbMATH00898043}, $JS$-metric space~\cite{zbMATH06433111}, non-triangular metric space~\cite{zbMATH07289654}, {\itshape et cetera}.
Despite the vast possibilities, we chiefly target the~\emph{$\ell^{p}$/$L^{p}$-type distances with~$p \in [0,1)$} including the \emph{$p$-Wasserstein distances} with either~$p \in [0,1)$ or with general cost function, as well as the \emph{divergences} ({\itshape e.g.} the Euclidean, Bregman, and Kullback-Leibler divergences) in information geometry.
These divergences are, in general, non-symmetric and no not enjoy the triangle inequality.

Next, we precise the convergence concepts of sequences with respect to a distance function.
\begin{definition}
  Suppose that~$d$ is a distance function on~$X$.
  A sequence~$(x_{i})$ in~$X$ is said to be
  \begin{itemize}[label=$\circ$, leftmargin=*]
    \item 
      \emph{right~$d$-convergent} to a \emph{right~$d$-limit}~$\bar{x} \in X$ if~$\lim_{i \to \infty} d(x,x_{i}) = 0$,
    \item 
      \emph{left~$d$-convergent} to a \emph{left~$d$-limit}~$\bar{x} \in X$ if~$\lim_{i \to \infty} d(x_{i},x) = 0$,
    \item 
      \emph{right~$d$-Cauchy} if~$\lim_{i \to \infty} \sup_{j \geq i} d(x_{j},x_{i}) = 0$,
    \item 
      \emph{left~$d$-Cauchy} if~$\lim_{i \to \infty} \sup_{j \geq i} d(x_{i},x_{j}) = 0$.
  \end{itemize}
\end{definition}
We could also specify a more strict behavior of a distance function regarding the habits of sequences.
\begin{definition}
  A distance function~$d$ on a set~$X$ is said to be 
  \begin{itemize}[label=$\circ$, leftmargin=*]
    \item 
      \emph{right Hausdorff} (resp. \emph{left Hausdorff}) if every right~$d$-convergent (resp. left~$d$-convergent) sequence in~$X$ has a unique limit,
    \item 
      \emph{right complete} (resp. \emph{left complete}) if every right~$d$-Cauchy (resp. left~$d$-Cauchy) sequence in~$X$ is right~$d$-convergent (resp. left~$d$-convergent).
  \end{itemize}
\end{definition}

Now we turn to continuity of functions with respect to a distance function.
\begin{definition}
  A function~$f : X \to \R \cup \{+\infty\}$ is said to be
  \begin{itemize}[label=$\circ$, leftmargin=*]
    \item 
      \emph{right sequentially~$d$-continuous} (resp. \emph{left sequentially~$d$-continuous}) at~$x \in X$ if~$\lim_{i \to \infty} f(x_{i}) = f(x)$ whenever~$(x_{i})$ is a sequence in~$X$ that is right (resp. left)~$d$-convergent to~$x$,
    \item 
      \emph{right sequentially~$d$-lower semicontinuous} (resp. \emph{left sequentially~$d$-lower semicontinuous}) at~$x \in X$ if
      \begin{align*}
        \liminf_{i \to \infty} f(x_{i}) \geq f(x)
      \end{align*}
      whenever~$(x_{i})$ is a sequence in~$X$ that is right (resp. left)~$d$-convergent to~$x$,
    \item 
      \emph{right sequentially~$d$-upper semicontinuous} (resp. \emph{left sequentially~$d$-upper semicontinuous}) at~$x \in X$ if
      \begin{align*}
        \limsup_{i \to \infty} f(x_{i}) \leq f(x)
      \end{align*}
      whenever~$(x_{i})$ is a sequence in~$X$ that is right (resp. left)~$d$-convergent to~$x$.
  \end{itemize}
\end{definition}

Obviously, every right (resp. left) sequentially~$d$-continuous function is right (resp. left) sequentially~$d$-lower semicontinuous.
For each fixed~$x \in X$, the functions~$d(x,\cdot)$ itself is right sequentially~$d$-continuous at~$x$, but there is no reason for a distance function to be either right or left sequentially~$d$-lower semicontinuous at other points.
Under this framework, it is even possible to impose continuity assumptions on a distance function~$d$ with respect to another distance function, say~$\rho$.

In this paper, we always require that~$d$ is \emph{separately right sequentially~$d$-lower semicontinuous}, {\itshape i.e.}~for each fixed point~$x \in X$, both functions~$d(x,\cdot)$ and~$d(\cdot,x)$ are right sequentially~$d$-lower semicontinuous at all points.

\begin{definition}
  A subset~$C \subset X$ is said to be 
  \begin{itemize}[label=$\circ$, leftmargin=*]
    \item 
      \emph{right sequentially~$d$-closed} if the limit of any right~$d$-convergent sequence in~$C$ is again in~$C$,
    \item 
      \emph{right sequentially~$d$-compact} if every sequence~$(x_{i})$ in~$C$ has a subsequence that is right~$d$-convergent to some point in~$C$.
  \end{itemize}
\end{definition}

The following lemma is classical in metric spaces.
Eventhough the proof in this setting is similar, we include it for the sake of completeness.
\begin{lemma}
  Suppose that~$d$ is a distance function on~$X$ and~$f : X \to \R \cup \{+\infty\}$ be a right sequentially~$d$-lower semicontinuous.
  Then the sublevel set
  \begin{align*}
    S_{f}(\lambda) := \{ x \in X \mid f(x) \leq \lambda \}
  \end{align*}
  is right sequentially~$d$-closed for all~$\lambda \in \R$.
\end{lemma}
\begin{proof}
  Suppose that~$\lambda \in \R$ and~$S_{f}(\lambda) \neq \emptyset$.
  Let~$(x_{i})$ be a sequence in~$S_{f}(\lambda)$ which is right~$d$-convergent to a point~$x \in X$.
  This implies
  \begin{align*}
    \lambda \geq \liminf_{i \to \infty} f(x_{i}) \geq f(x)
  \end{align*}
  and hence~$x \in S_{f}(\lambda)$.
  Therefore~$S_{f}(\lambda)$ is right sequentially~$d$-closed.
\end{proof}

We also need the next lemma which is a generalzation of the Cantor's intersection theorem.
To prove this, we first require the definition of balls.
The \emph{right~$d$-ball} about the point~$x \in X$ with radius~$r > 0$ is defined by
\begin{align*}
  B(x,r) := \{ y \in X \mid d(y,x) < r \}.
\end{align*}

\begin{lemma}\label{lem: Cantor intersection theorem}
  Let~$d$ be a right complete, right Hausdorff distance function on~$X$
  % Let~$(X,d)$ be a complete non-triangular metric space 
  and let~$(S_{i})_{i=0}^{\infty}$ be a decreasing sequence of nonempty right sequentially~$d$-closed subsets of~$X$.
  Suppose that for each~$i \in \N\cup\{0\}$, there exist~$x_{i} \in S_{i}$ and~$r_{i} > 0$ in which~$S_{i} \subset B(x_{i},r_{i})$.
  If~$r_{i} \to 0$, then the sequence~$(x_{i})$ is right~$d$-convergent to a point~$x^{\ast}$ and~$\bigcap_{i=0}^{\infty} S_{i} = \{x^{\ast}\}$.
\end{lemma}
\begin{proof}
  Due to the construction, it holds~$d(x_{j}, x_{i}) \leq r_{i}$ for any~$i \in \N$ and any~$j \geq i$. 
  We then have~$\lim_{i \to \infty} \sup_{j \geq i} d(x_{j},x_{i}) = 0$, where the completeness of~$X$ implies that~$(x_{i})$ is right~$d$-convergent to a unique point~$x^{\ast} \in X$.
  Since~$x_{j} \in S_{i}$ for all~$j \geq i$ and~$S_{i}$ is closed for all~$i\in \N$, it holds that the limit~$x^{\ast}$ belongs to all~$S_{i}$'s.
  This shows that~$x^{\ast} \in \bigcap_{i=0}^{\infty} S_{i}$.
  
  Next, we prove that~$\bigcap_{i=0}^{\infty} S_{i}$ contains only~$x^{\ast}$.
  Suppose that~$y^{\ast} \in \bigcap_{i=0}^{\infty} S_{i}$.
  Subsequently we have~$y^{\ast} \in S_{i} \subset B(x_{i},r_{i})$ for every~$i \in \N$.
  Passing to the limit, we obtain~$\lim_{i \to \infty} d(y^{\ast},x_{i}) = 0$ and the uniqueness of a limit implies that~$y^{\ast} = x^{\ast}$.
  We hence conclude that~$\bigcap_{i=0}^{\infty} S_{i} = \{x^{\ast}\}$.
\end{proof}

\section{Borwein-Preiss variational principle}

In this section, we state and prove the Borwein-Preiss variational principle using non-triangular distance function.
The proof of the main theorem, shown below, will also be the backbone of that of the next section.

\begin{theorem}[Borwein-Preiss variational principle]\label{thm: bpvp}
  Let~$d$ be a right complete, right Hausdorff, and separately right sequentially~$d$-lower semicontinuous distance function on~$X$, and let~$f : X \to \R \cup \{+\infty\}$ be a right sequentially~$d$-lower semicontinuous function that is bounded from below.
  Suppose that~$\varepsilon > 0$ and~$z_{0} \in X$ satisfy
  \[
    f(z_{0}) < \inf_{X} f + \varepsilon.
  \]
  Let~$(\delta_{i})_{i=0}^{\infty}$ be a sequence of positive reals.
  Then, there exist~$\bar{z} \in X$ and a sequence~$(z_{i})_{i=0}^{\infty}$ in~$X$ that is right~$d$-convergent to~$\bar{z}$ such that
  \begin{enumerate}[label=\upshape(\alph*), leftmargin=*]
    \item\label{ccs: bpvp1}
      $d(\bar{z},z_{i}) \leq \frac{\varepsilon}{2^{i}\delta_{0}}$ for all~$i \in \N$,
    \item\label{ccs: bpvp2} 
      $f(\bar{z}) + \sum_{k=0}^{\infty} \delta_{k}d(\bar{z},z_{k}) \leq f(z_{0})$,
      \item\label{ccs: bpvp3} 
      $f(z) + \sum_{k=0}^{\infty} \delta_{k} d(z,z_{k}) > f(\bar{z}) + \sum_{k=0}^{\infty} \delta_{k} d(\bar{z},z_{k})$, for all~$z \in X \setminus \{\bar{z}\}$.
  \end{enumerate}
\end{theorem}
\begin{proof}
  We begin by constructing the sequence~$(z_{i})$.
  First, we set
  \[
    S_{0} := \{ z \in X \mid f(z) + \delta_{0} d(z,z_{0}) \leq f(z_{0}) \}.
  \]
  Observe that~$z_{0} \in S_{0}$ and since~$f$ and~$d(\cdot,z_{0})$ are right sequentially~$d$-lower semicontinuous, the set~$S_{0}$ is right sequentially~$d$-closed.
  % Note that
  % \[
  %   \delta_{0} d(z,z_{0}) \leq f(z_{0}) - f(z) \leq f(z_{0}) - \inf_{X} f \leq \varepsilon
  % \]
  % for any~$z \in S_{0}$.
  Next, take~$z_{1} \in S_{0}$ such that
  \[
    f(z_{1}) + \delta_{0} d(z_{1},z_{0}) \leq \inf_{z \in S_{0}} [f(z) + \delta_{0} d(z,z_{0})] + \frac{\varepsilon \delta_{1}}{2\delta_{0}}
  \]
  and define
  \[
    S_{1} := \{ z \in S_{0} \mid f(z) + \delta_{0} d(z,z_{0}) + \delta_{1} d(z,z_{1}) \leq f(z_{1}) + \delta_{0} d(z_{1},z_{0}) \}.
  \]
  Again, we have~$z_{1} \in S_{1}$ and~$S_{1}$ is right sequentially~$d$-closed.
  Next, suppose that we have already defined~$z_{j} \in X$ and~$S_{j} \subset S_{j-1}$ for~$j = 1,2,\dots,i-1$ in such a way that
  \[
    f(z_{j}) + \sum_{k=0}^{j-1} \delta_{k} d(z_{j},z_{k}) \leq \inf_{z \in S_{j-1}} \left[f(z) + \sum_{k=0}^{j-1} \delta_{k} d(z,z_{k})\right] + \frac{\varepsilon \delta_{j}}{2^{j}\delta_{0}}
  \]
  and
  \[
    S_{j} := \left\{ z \in S_{j-1} \mid f(z) + \sum_{k=0}^{j} \delta_{k} d(z,z_{k}) \leq f(z_{j}) + \sum_{k=0}^{j-1} \delta_{k} d(z_{j},z_{k}) \right\}.
  \]
  We proceed by choosing~$z_{i} \in S_{i-1}$ satisfying
  \[
    f(z_{i}) + \sum_{k=0}^{i-1} \delta_{k} d(z_{i},z_{k}) \leq \inf_{z \in S_{i-1}} \left[ f(z) + \sum_{k=0}^{i-1} \delta_{k} d(z,z_{k}) \right] + \frac{\varepsilon \delta_{i}}{2^{i}\delta_{0}}
  \]
  and subsequently put
  \[
    S_{i} := \left\{ z \in S_{i-1} \mid f(z) + \sum_{k=0}^{i} \delta_{k} d(z,z_{k}) \leq f(z_{i}) + \sum_{k=0}^{i-1} \delta_{k} d(z_{i},z_{k}) \right\}.
  \]
  This inductive process defines a sequence~$(z_{i})$ in~$X$ and a decreasing sequence~$(S_{i})$ of right sequentially~$d$-closed subsets of~$X$, and that~$z_{i} \in S_{i}$ for all~$i \in \N\cup\{0\}$.

  Now, we are going to show~\ref{ccs: bpvp1}.
  Note that for any~$z \in S_{i}$, it holds that
  \begin{align*}
    \delta_{i} d(z,z_{i}) 
      &\leq \left[ f(z_{i}) + \sum_{k=0}^{i-1} \delta_{k}d(z_{i},z_{k}) \right] - \left[ f(z) + \sum_{k=0}^{i-1} \delta_{k}d(z,z_{k}) \right] \\
      &\leq \left[ f(z_{i}) + \sum_{k=0}^{i-1} \delta_{k}d(z_{i},z_{k}) \right] - \inf_{z \in S_{i-1}}\left[ f(z) + \sum_{k=0}^{i-1} \delta_{k}d(z,z_{k}) \right] \\
      &\leq \frac{\varepsilon \delta_{i}}{2^{i}\delta_{0}}.
  \end{align*}
  Hence we have
  \begin{align}\label{eqn: sequence estimation}
    d(z,z_{i}) \leq \frac{\varepsilon}{2^{i}\delta_{0}}
  \end{align}
  for any~$z \in S_{i}$ and any~$i \in \N \cup \{0\}$.
  In other words, we have~$S_{i} \subset B(z_{i},r_{i})$ with~$r_{i} := \frac{\varepsilon}{2^{i}\delta_{0}}$ for all~$i \in \N \cup \{0\}$.
  Lemma~\ref{lem: Cantor intersection theorem} shows that~$(z_{i})$ is right~$d$-convergent to a unique point~$\bar{z}$ and~$\bigcap_{i=0}^{\infty} S_{i} = \{\bar{z}\}$.
  In view of~\eqref{eqn: sequence estimation}, the conclusion~\ref{ccs: bpvp1} is proved.

  Next we prove~\ref{ccs: bpvp2}. 
  One may observe for any~$q \geq j$ that
  \begin{align*}
    f(z_{0})
      &\geq f(z_{j}) + \sum_{k=0}^{j-1} \delta_{k} d(z_{j},z_{k}) \\
      &\geq f(z_{q}) + \sum_{k=0}^{q-1} \delta_{k} d(z_{q},z_{k}) \\
      &\geq f(\bar{z}) + \sum_{k=0}^{q} \delta_{k} d(\bar{z},z_{k}).
  \end{align*}
  The desired assertion follows by passing~$q \to \infty$ in the above inequalities.

  We finally prove~\ref{ccs: bpvp3}.
  Suppose that~$z \in X\setminus\{\bar{z}\}$.
  This means~$z \not\in \bigcap_{i=0}^{\infty} S_{i}$, and hence the inequalities
  \begin{align*}
    f(z) + \sum_{k=0}^{j} \delta_{k} d(z,z_{k}) 
    &> f(z_{j}) + \sum_{k=0}^{j-1} \delta_{k} d(z_{j},z_{k}) \\
    &\geq f(\bar{z}) + \sum_{k=0}^{j} \delta_{k} d(\bar{z},z_{k})
  \end{align*}
  hold for sufficiently large~$j \in \N \cup \{0\}$.
  Letting~$j \to \infty$ in the above inequalities, the conclusion~\ref{ccs: bpvp3} is verified.
\end{proof}

We could also deduce a simpler statement of the above theorem, which is known in the literature as the \emph{weak form} of the Borwein-Preiss variational principle.
\begin{theorem}[Borwein-Preiss variational principle: Weak form]\label{thm: bpvp weak form}
  Let~$d$ be a right complete, right Hausdorff, and separately right sequentially~$d$-lower semicontinuous distance function on~$X$, and let~$f : X \to \R \cup \{+\infty\}$ be a right sequentially~$d$-lower semicontinuous function that is bounded from below.
  For any sequence~$(\delta_{i})_{i=0}^{\infty}$ of positive reals, there exists a sequence~$(z_{i})_{i=0}^{\infty}$ in~$X$ that is right~$d$-convergent to~$\bar{z} \in X$ such that the perturbed function
  \begin{align*}
    f(\cdot) + \sum_{k=0}^{\infty} \delta_{k} d(\cdot,z_{k}) : X \to \R \cup \{+\infty\}
  \end{align*}
  has a unique minimizer~$\bar{z}$.
\end{theorem}

\section{Ekeland variational principle}

In this section, we state and prove the Ekeland variational principle where the kernel function is a non-triangular distance.
A big difference, and an obstacle, of deducing the Ekeland variational principle from the Borwein-Preiss theorem is the positivity of the sequence~$(\delta_{i})$.
Although there were some refined versions (see, for example \cite{zbMATH01496358, zbMATH06522706}), we would like to keep the presentation simple and to focus merely on deducing the Ekeland variational principle.
The proof exploits a similar technique that is used in showing Theorem~\ref{thm: bpvp weak form}, but with a slight adaptation.
Also, one should note that the following theorem requires a stronger continuity assumption on the distance function~$d$.

\begin{theorem}[Ekeland variational principle]\label{thm: evp}
  Let~$d$ be a distance function on~$X$ which is right complete, right Hausdorff, right sequentially~$d$-lower semicontinuous in the first argument, and right sequentially~$d$-continuous in the second argument.
  Let~$f : X \to \R \cup \{+\infty\}$ be a right sequentially~$d$-lower semicontinuous function that is bounded from below.
  Suppose that~$\varepsilon > 0$ and~$z_{0} \in X$ satisfy
  \begin{align*}
    f(z_{0}) < \inf_{X} f + \varepsilon.
  \end{align*}
  Then, there exists~$\bar{z} \in X$ such that
  \begin{enumerate}[label=\upshape(\alph*), leftmargin=*]
    \item\label{ccs: evp1}
      $d(\bar{z},z_{0}) \leq 1$
    \item\label{ccs: evp2} 
      $f(\bar{z}) + \varepsilon d(\bar{z},z_{0}) \leq f(z_{0})$
      \item\label{ccs: evp3} 
        $f(z) + \varepsilon d(z,\bar{z}) \geq f(\bar{z})$ for all~$z \in X$.
  \end{enumerate}
\end{theorem}
\begin{proof}
  We first define
  \begin{align*}
    S_{0} := \{ z \in X \mid f(z) + \varepsilon d(z,z_{0}) \leq f(z_{0}) \}.
  \end{align*}
  Note that~$S_{0} \neq \emptyset$ and since~$f$ and~$d(\cdot,z_{0})$ are right sequentially~$d$-lower semicontinuous, the set~$S_{0}$ is right sequentially~$d$-closed.
  Choose~$z_{1} \in S_{0}$ such that
  \begin{align*}
    f(z_{1}) \leq \inf_{z \in S_{0}} f(z) + \frac{\delta_{0}}{2}.
  \end{align*}
  and define
  \begin{align*}
    S_{1} := \{ z \in X \mid f(z) + \varepsilon d(z,z_{1}) \leq f(z_{1}) \}.
  \end{align*}
  Suppose that we have already defined~$z_{j} \in S_{j-1}$ and~$S_{j} \subset S_{j-1}$ for~$j = 1,\dots,i-1$.
  We then choose~$z_{i} \in S_{i-1}$ satisfying
  \begin{align*}
    f(z_{i}) \leq \inf_{z \in S_{i-1}} f(z) + \frac{\varepsilon}{2^{i}}
  \end{align*}
  and set
  \begin{align*}
    S_{i} := \{ z \in X \mid f(z) + \varepsilon d(z,z_{i}) \leq f(z_{i}) \}.
  \end{align*}
  This inductive process defines a sequence~$(z_{i})$ in~$X$ and a decreasing sequence~$(S_{i})$ of right sequentially~$d$-closed subsets of~$X$ in such a way that~$z_{i} \in S_{i}$ for every~$i \in \N \cup \{0\}$.

  For any~$z \in S_{i}$, we may see that
  \begin{align}\label{eqn: css1 leading ineq}
    \varepsilon d(z,z_{i})
      &\leq f(z_{i}) - f(z) \leq f(z_{i}) - \inf_{z \in S_{i-1}} f(z) \leq \frac{\varepsilon}{2^{i}}.
  \end{align}
  This shows that~$S_{i} \subset B(z_{i},r_{i})$ with~$r_{i} := \frac{\varepsilon}{2^{i}}$, for all~$i \in \N \cup \{0\}$.
  Following from~Lemma~\ref{lem: Cantor intersection theorem}, the sequence~$(z_{i})$ is right~$d$-convergent to a unique point~$\bar{z} \in X$ and that~$\bigcap_{i=0}^{\infty} S_{i} = \{\bar{z}\}$.
  The inequalities~\eqref{eqn: css1 leading ineq} also guarantees~\ref{ccs: evp1} by letting~$i = 0$.
  The fact that~$\bar{z} \in S_{0}$ yields~\ref{ccs: evp2}.
  
  Finally, take any~$z \in X \setminus \{\bar{z}\}$.
  This means~$z \not\in \bigcap_{i=0}^{\infty} S_{i}$ and hence~$z \not\in S_{j}$ for all sufficiently large~$j \in \N$.
  For such~$j \in \N$, we have
  \begin{align*}
    f(z) + \varepsilon d(z,z_{j}) > f(z_{j}) \geq f(\bar{z}) + \varepsilon d(\bar{z},z_{j}).
  \end{align*}
  Since~$d$ is right sequentially~$d$-continuous in the second argument, letting~$j \to \infty$ gives~$f(z) + \varepsilon d(z,\bar{z}) \geq f(\bar{z})$ and the theorem is thus proved.
\end{proof}

Again, we present the weak form of the above theorem.
\begin{theorem}[Ekeland variational principle: Weak form]\label{thm: evp weak}
  Let~$d$ be a distance function on~$X$ which is right complete, right Hausdorff, right sequentially~$d$-continuous in the first argument, and right sequentially~$d$-lower semicontinuous in the second argument.
  Let~$f : X \to \R \cup \{+\infty\}$ be a right sequentially~$d$-lower semicontinuous function that is bounded from below.
  For any~$\varepsilon > 0$, there exists~$\bar{z} \in X$ such that the following perturbed function
  \begin{align*}
    f(\cdot) + \varepsilon d(\cdot,\bar{z}) : X \to \R\cup\{+\infty\}
  \end{align*}
  has a minimizer~$\bar{z}$.
\end{theorem}

\section{Caristi theorem}

In this section, we provide a useful application of the Ekeland variational principle to prove the Caristi fixed point theorem for a set-valued map.
The original Caristi theorem~\cite{zbMATH03477902} is known to be equivalent to the Ekeland variational principle in metric spaces~\cite{zbMATH03574569}.
Once again, one should note that a stronger hypothesis about the continuity of the distance function~$d$ is used.
Moreover, the trick to construct the distance function~$\rho$ on~$X \times X$ that swaps the order as in~\eqref{eqn: rho swaps d} is crucial for the proof of the next theorem.

\begin{theorem}
  Let~$d$ be a distance function on~$X$ which is right complete, right Hausdorff, and jointly sequentially~$d$-continuous.
  Let~$\varphi : X \to \R \cup \{+\infty\}$ be a right sequentially~$d$-lower semicontinuous function bounded below and~$T : X \multimap X$ be a set-valued map such that
  \begin{align*}
    \varphi(y) \leq \varphi(x) - d(x,y) + \iota_{\gr T}(x,y)
  \end{align*}
  holds for all~$x,y \in X$.
  If the graph~$\gr T$ is right sequentially~$\rho$-closed in~$X \times X$, where~$\rho$ is a distance function on~$X \times X$ defined by
  \begin{align}\label{eqn: rho swaps d}
    \rho((x,y),(x',y')) := d(x',x) + d(y',y).
  \end{align}
  Then~$T$ has a fixed point, {\itshape i.e.}~$\fix(T) := \{x \in X \mid x \in T(x)\} \neq \emptyset$.
  Moreover these fixed points are end points, {\itshape i.e.}~$\fix(T) = \Endpoints(T) := \{x \in X \mid T(x) = \{x\} \}$.
\end{theorem}
\begin{proof}
  Let~$\varepsilon \in (0,1/2)$ and define a function~$f : X \times X \to \R \cup \{+\infty\}$ by
  \begin{align*}
    f(x,y) := \varphi(x) - (1-\varepsilon)d(x,y) + \iota_{\gr T}(x,y).
  \end{align*}
  Following from the continuity assumptions of~$f$ and~$d$ as well as the closedness assumption of~$\gr T$, the function~$f$ is right sequentially~$d$-lower semicontinuous.
  From the assumptions on~$d$, we could see that~$\rho$ satisfies the hypotheses of Theorem~\ref{thm: evp weak}.
  Applying Theorem~\ref{thm: evp weak}, there exists~$(\bar{x},\bar{y}) \in \gr T$ such that
  \begin{align}\label{eqn: caristi evp min}
    f(\bar{x},\bar{y}) \leq f(x,y) + \varepsilon \rho((x,y),(\bar{x},\bar{y}))
  \end{align}
  for every~$(x,y) \in \gr T$.
  Now take~$\bar{z} \in T(\bar{y})$ and~\eqref{eqn: caristi evp min} implies
  \begin{align*}
    \varphi(\bar{x}) - (1-\varepsilon)d(\bar{x},\bar{y})
      &\leq \varphi(\bar{y}) - (1-\varepsilon)d(\bar{y},\bar{z}) + \varepsilon \rho((\bar{y},\bar{z}),(\bar{x},\bar{y})) \\
      & = \varphi(\bar{y}) - (1-\varepsilon)d(\bar{y},\bar{z}) + \varepsilon [d(\bar{x},\bar{y}) + d(\bar{y},\bar{z})].
  \end{align*}
  Rearranging yields
  \begin{align*}
    0 \leq \varphi(\bar{x}) - \varphi(\bar{y}) - d(\bar{x},\bar{y}) \leq -(1-2\varepsilon)d(\bar{y},\bar{z}).
  \end{align*}
  It must be the case that~$\bar{y} = \bar{z}$ and hence~$\bar{y} \in T(\bar{y})$, proving that~$\fix(T) \neq \emptyset$.
\end{proof}

\section{Equilibrium problems}

The solution of an equilibrium problem associated to a bifunction~$F : X \times X \to \R$ over a constraint set~$C \subset X$ is any point $\bar{x} \in C$ such that
\begin{align}\label{eqn: equilibrium problem}
  F(\bar{x},y) \geq 0 \quad (\forall y \in C).
\end{align}
Let us write~$\EP(F,C) := \{x \in C \mid F(x,y) \geq 0 \quad (\forall y \in C) \}$ to denote the set of all solutions of the above problem.
The problem itself was originated in~\cite{zbMATH03472891} (see also \cite{zbMATH01090020}).
It is possible to show an existence of a solution of an equilibrium problem by applying Ekeland variational principle, when the value of the associated bifunction could be appropriately lower estimated~\cite{zbMATH07347459}.

Shown in the following is an existence theorem for an equilibrium problem~\eqref{eqn: equilibrium problem}.
Note yet again the strengthened continuity of the distance function.
\begin{theorem}
  Let~$d$ be a distance function on~$X$ which is right complete, right Hausdorff, right sequentially~$d$-continuous separately in the both argument.
  Also let~$F : X \times X \to \R$ be a bifunction which is right sequentially~$d$-upper semicontinuous in its first argument, and that has a lower estimate 
  \begin{align}\label{eqn: lower estimate}
    F(x,y) \geq \varphi(y) - \varphi(x), \quad (\forall y \in X).
  \end{align}
  for some right sequentially~$d$-lower semicontinuous function~$\varphi : X \to \R$ which is bounded below.
  If~$X$ is right sequentially~$d$-compact, then~$\EP(F,X) \neq \emptyset$.
\end{theorem}
\begin{proof}
  Let~$(\varepsilon_{i})$ be a sequence of positive reals with~$\lim_{i \to \infty} \varepsilon_{i} = 0$.
  For each~$i \in \N$, Theorem~\ref{thm: evp weak} and~\eqref{eqn: lower estimate} give an existence of a point~$x_{i} \in X$ such that
  \begin{align}\label{eqn: almost EP}
    F(x_{i},y) \geq \varphi(y) - \varphi(x_{i}) \geq -\varepsilon d(y,x_{i})
  \end{align}
  for any~$y \in X$.
  Since~$X$ is right sequentially~$d$-compact, the sequence~$(x_{i})$ has a subsequence~$(x_{i_{k}})$ that is right~$d$-convergent to some point~$\bar{x}$.
  With the continuity assumptions of~$F$ and~$d$, we obtain by passing the limit in~\eqref{eqn: almost EP} that
  \begin{align*}
    F(\bar{x},y) \geq 0
  \end{align*}
  for all~$y \in X$.
  Therefore we have a conclusion that~$\EP(F,X) \neq \emptyset$.
\end{proof}

\section*{Conclusion and remarks}

The main results in this papers are the two variational principles, the types of Borwen-Preiss and of Ekeland.
The versions presented here exploit only on the sequential convergence and continuity assumptions of a distance function, which is not required to be symmetric nor to enjoy the triangle inequality.
Two applications, to fixed point theory and equilibrium problem, are also presented to justify that the theorem could be helpful in generalizing the known scope of results.
One could see that different continuity assumptions on the distance functions are required in each of the sections.

It is important to note that every non-symmetric distance function~$d$ could be \emph{symmetrized} by considering the distance function~$\tilde{d}(x,y) = d(x,y) + d(y,x)$ (one could also put weights to each term~$d(x,y)$ and~$d(y,x)$).
Then the convergent sequences under this symmetrized distance~$\tilde{d}$ are exactly those which are both right and left~$d$-convergent.

\section*{Acknowledgments}
The authors acknowledge the financial support provided by the Center of Excellence in Theoretical and Computational Science (TaCS-CoE), KMUTT.
The first author was supported by the Petchra Pra Jom Klao Ph.D. Research Scholarship from King Mongkut's University of Technology Thonburi (Grant No.69/2563).

\section*{Conflicts of interests}
The authors state no conflicts of interests.

\section*{Author contributions}
\textbf{Conceptualization}, Parin Chaipunya and Poom Kumam;
\textbf{Supervision}, Parin Chaipunya;
\textbf{Investigation}, Natthaya Boonyam and Parin Chaipunya;
\textbf{Writing - original draft}, Natthaya Boonyam and Parin Chaipunya; 
\textbf{Review \& editing}, Parin Chaipunya and Poom Kumam.

\small
\renewcommand\bibname{References}

% \bibliographystyle{abbrvnat}
% \bibliography{variationalprincip}

\end{document}